\documentclass[sn-mathphys,Numbered]{sn-jnl}

\usepackage{graphicx}%
\usepackage{multirow}%
\usepackage{amsmath,amssymb,amsfonts}%
\usepackage{amsthm}%
\usepackage{mathrsfs}%
\usepackage[title]{appendix}%
\usepackage{xcolor}%
\usepackage{textcomp}%
\usepackage{manyfoot}%
\usepackage{booktabs}%
\usepackage{algorithm}%
\usepackage{algorithmicx}%
\usepackage{algpseudocode}%
\usepackage{listings}%

\theoremstyle{thmstyleone}%
\newtheorem{theorem}{Theorem}
\newtheorem{proposition}[theorem]{Proposition}%

\theoremstyle{thmstyletwo}%
\newtheorem{example}{Example}%
\newtheorem{remark}{Remark}%
\newtheorem{corollary}{Corollary}%
\newtheorem{lemma}{Lemma}%

\theoremstyle{thmstylethree}%

\begin{document}

\title[Article Title]{On semi-continuity and continuity of the smallest and largest minimizing point of real convex functions with applications in probability and statistics}


\author{\fnm{Dietmar} \sur{Ferger}}\email{dietmar.ferger@tu-dresden.de}



\affil{\orgdiv{Fakult\"{a}t Mathematik}, \orgname{Technische Universit\"{a}t Dresden}, \orgaddress{\street{Zellescher Weg 12-14}, \city{Dresden}, \postcode{01069}, \country{Germany}}}




\abstract{We prove that the smallest minimizer $\sigma(f)$ of a real convex function $f$ is less than or equal to a real point $x$ if and only if
the right derivative of $f$ at $x$ is non-negative. Similarly, the largest minimizer $\tau(f)$ is greater or equal to $x$ if and only if the left derivative
of $f$ at $x$ is non-positive. From this simple result we deduce measurability and semi-continuity of the functionals $\sigma$ and $\tau$.
Furthermore, if $f$ has a unique minimizing point, so that $\sigma(f)=\tau(f)$, then the functional is continuous at $f$. With these analytical preparations we can
apply Continuous Mapping Theorems to obtain several Argmin theorems for convex stochastic processes. The novelty here are statements about classical distributional convergence and almost sure convergence, if the limit process does not have a unique minimum point. This is possible by replacing the natural topology on $\mathbb{R}$ with the order topologies. Another new feature is that not only sequences but more generally nets of convex stochastic processes are allowed.}

\keywords{order topologies, distributional convergence in topological spaces, Argmin theorems, functional limit theorem for convex processes}



\maketitle

\section{Introduction}\label{sec1}

Let $f:\mathbb{R} \rightarrow \mathbb{R}$ be a convex function with pertaining set
\begin{equation} \label{A}
A(f):=\{t \in \mathbb{R}: f(t)=\inf_{s \in \mathbb{R}}f(s)\}
\end{equation}
of all minimizing points. This \emph{minimum set} can equivalently be described in terms of the right and left derivative $D^+f$ and $D^-f$ of $f$.
Indeed, it follows from Theorem 23.2 of Rockafellar \cite{Rockafellar} that
\begin{equation} \label{minset}
 A(f)= \{t \in \mathbb{R}: D^-f(t) \le 0 \le D^+f(t)\}.
\end{equation}

Of course, it can happen that $A(f)$ is empty. However, if this is not the case, it is well-known (and actually easy to see) that then
$A(f)$ is closed and convex and hence is a closed interval. The extreme case $A(f)=\mathbb{R}$ occurs if and only if $f$ is a constant function. So, as long as $f$ is not a constant function there are three
possibilities: (i) $A(f)=[a,b], a \le b \in \mathbb{R}$. (ii) $A(f)= [a,\infty), a \in \mathbb{R}$. (iii) $A(f)=(-\infty,b], b \in \mathbb{R}$.\\

Let $C:=\{f:\mathbb{R} \rightarrow \mathbb{R}; f \text{ convex}\}$ be the class of all convex functions.
Introduce
\begin{equation*}
 S:=\{f \in C: \inf A(f)>-\infty\}
\end{equation*}
\begin{equation*}
 S':=\{f \in C: \sup A(f) < \infty\}.
\end{equation*}

Thus $S$ and $S'$ consist exactly of those functions for which the smallest and largest minimizer, respectively, exist.
Consequently the functionals $\sigma:S \rightarrow \mathbb{R}$ and $\tau:S' \rightarrow \mathbb{R}$
given by $\sigma(f)=\min A(f)$ and $\tau(f)=\max A(f)$ are well defined on their domains.\\

Our first main result provides a necessary and sufficient condition for the location of the smallest and the largest minimizer, respectively.
It is rather simple with an elementary proof, but with it we will draw a whole series of useful conclusions.
In section 2 it is shown that $\sigma$ and $\tau$ are measurable and semi-continuous. On the set $S_u$ of all convex functions with exactly one minimizer the functionals $\sigma$ and $\tau$ coincide and are furthermore continuous there. In section 3 this is used in combination with Continuous Mapping Theorems to derive several Argmin theorems for convex stochastic processes. For a further discussion of our findings we refer to the concluding remarks at the end of section 3.\\

\begin{theorem} \label{location}
Let $\sigma(f)=\min A(f)$ and $\tau(f)=\max A(f)$ be the smallest and largest minimizing point of the convex function $f$.
Then the following equivalent relations hold for every $x \in \mathbb{R}$:

\begin{equation} \label{sigma}
\sigma(f) \le x \quad \Longleftrightarrow \quad D^+f(x) \ge 0
\end{equation}
\begin{equation} \label{tau}
\tau(f) \ge x  \quad \Longleftrightarrow \quad D^-f(x) \le 0.
\end{equation}
\end{theorem}

\begin{proof} For the proof of (\ref{sigma}) we briefly write $\sigma$ for $\sigma(f)$. Recall that $D^+f$ and $D^-f$ are non-decreasing, confer, e.g., Theorem 1.3.3 in Niculescu and Persson \cite{Niculescu}. Thus $\sigma \le x$ entails $D^+f(x) \ge D^+f(\sigma) \ge 0$, where the last inequality follows from (\ref{minset}), because
$\sigma \in A(f)$. To see the reverse conclusion in (\ref{sigma}) we use that by Theorem 1.3.1 in Niculescu and Persson \cite{Niculescu} the
difference quotients $\frac{f(t)-f(x)}{t-x}$ are non-increasing as $t\downarrow x$. Therefore we obtain
\begin{equation} \label{D+}
 0 \le D^+f(x)=\inf_{t>x}\frac{f(t)-f(x)}{t-x} \le \frac{f(t)-f(x)}{t-x} \quad \forall \; t>x.
\end{equation}
Multiplication with the positive difference $t-x$ yields that $f(t) \ge f(x)$ for all $t>x$ and so
$$
 \inf_{t>x}f(t) \ge f(x) \ge \inf_{t \le x}f(t),
$$
where the second inequality is trivial. To sum up we arrive at
\begin{equation} \label{inf}
 \inf_{t \in \mathbb{R}} f(t)=\inf_{t \le x} f(t).
\end{equation}
Next we consider a point $t < \sigma$. Infer from  the monotonicity of $D^-f$ that
$D^-f(t) \le D^-f(\sigma) \le 0$, where the last inequality is ensured by (\ref{minset}).
Just as for $D^+f$, we see that $D^-f(t)=\sup_{s<t}\frac{f(s)-f(t)}{s-t}$, whence in view of $D^-f(t) \le 0$ we have that
$$
 \frac{f(s)-f(t)}{s-t} \le 0 \quad \forall\; s<t.
$$
Multiplication with the negative difference $s-t$ gives $f(s) \ge f(t)$ for all $s <t < \sigma$.
Since $f$ is continuous,
taking the limit $t \uparrow \sigma$ finally shows that $f$ is non-increasing on the
closed half-line $(-\infty,\sigma]$. In particular, $f(t) \ge f(\sigma)$ for all $t \le \sigma$.
Actually, on the open interval $(-\infty,\sigma)$ the inequality is strict:
\begin{equation} \label{strict}
 f(t) > f(\sigma) \quad \forall \; t < \sigma.
\end{equation}
This is because otherwise there exists some $t_0 < \sigma$ such that $f(t_0) \le f(\sigma)$.
Since $\sigma$ is a minimizing point, $t_0$ must be a minimizing point as well.
This is a contradiction to minimality of $\sigma$.

Now, assume that $x< \sigma$. Deduce from $f$ is non-increasing on $(-\infty,x] \subseteq (-\infty,\sigma]$
that $f(t) \ge f(x)$ for all $t \le x$. Consequently, $\inf_{t \le x}f(t) \ge f(x)$ and therefore by (\ref{inf}) and (\ref{strict}):
$$
 f(\sigma)=\inf_{t \in \mathbb{R}} f(t) = \inf_{t \le x} f(t) \ge f(x) > f(\sigma).
$$
This is a contradiction and thus $x \ge \sigma$ is true as desired.

For the proof of (\ref{tau}) we use a time-reversing argument. Introduce the function $f_{-}$ defined by $f_{-}(t):=f(-t), t \in \mathbb{R}$.
One easily verifies that $f_{-}$ is convex and that $\tau(f)=-\sigma(f_{-})$. We obtain by (\ref{sigma}):
$$
 \tau(f) \ge x \quad \Longleftrightarrow \quad -\sigma(f_{-}) \ge x \quad \Longleftrightarrow \quad \sigma(f_{-}) \le -x \quad \Longleftrightarrow \quad D^+f_{-}(-x) \ge 0
$$
and the assertion follows upon noticing that $D^+f_{-}(-x)=-D^-f(x)$.
\end{proof}

\begin{remark} \label{exist}
Under the existence of the minimizers one has that:
\begin{equation} \label{sim}
 (a) \quad \sigma(f)\le x, \tau(f) \ge x \quad \Longleftrightarrow \quad (b) \quad D^+f(x) \ge 0, D^-f(x) \le 0.
\end{equation}
Indeed, recall that $D^+f$ and $D^-f$ are non-decreasing, which by (\ref{minset}) shows necessity of (b). By another application of (\ref{minset}) the point x in (b) is a minimizer, $x \in A(f)$, and (a) follows using the minimum and maximum property of $\sigma(f)=\min A(f)$ and $\tau(f)=\max A(f)$, respectively.

Note that we cannot readily infer Proposition \ref{location}  from (\ref{sim}). For example, if we only know that $D^+f(x) \ge 0$ holds, then $x$ need not be a minimizing point in general, and is not, as can be seen from simple examples as for instance $f(x)=x^2$. Consequently, the argument via the minimum property of $\sigma(f)$ fails.\\
\end{remark}

For every non-decreasing function $F:\mathbb{R} \rightarrow \mathbb{R}$ we introduce the \emph{generalized inverses}:
$$
 F^\wedge(y):=\inf\{x \in \mathbb{R}: F(x) \ge y \} \quad \text{and} \quad F^\vee(y):=\sup\{x \in \mathbb{R}: F(x) \le y \}, \; y \in \mathbb{R}.
$$
For properties of these inverses, confer Embrechts and Hofert \cite{Embrechts}, Feng et.al. \cite{Feng} or Fortelle \cite{Fortelle}.
Notice that (\ref{sigma}) is the same as $[\sigma(f),\infty)= \{x \in \mathbb{R}: D^+f(x) \ge 0 \}$ and hence
$$
 \sigma(f)= (D^+f)^\wedge(0).
$$

\vspace{0.5cm}
Similarly, $(-\infty,\tau(f)]=\{x \in \mathbb{R}: D^-f(x) \le 0 \}$ by (\ref{tau}), whence
$$
 \tau(f)= (D^-f)^\vee(0).
$$

\section{Measurability, Semi-continuity and Continuity of the argmin-functionals}

Recall that $C$ is the class of all convex functions on $\mathbb{R}$.
For each $t \in \mathbb{R}$ let $\pi_t:C \rightarrow \mathbb{R}$ denote the projection (evaluation map)
at $t$, that is $\pi_t(f)=f(t)$. We endow the function space $C$ with the $\sigma-$algebra generated by the projections:
$\mathcal{C}:= \sigma(\pi_t: t \in \mathbb{R})$. Furthermore, $C$ is equipped with the topology $\mathcal{T}$ of pointwise convergence, which is known to be generated by the projections: $\mathcal{T}:=\tau(\pi_t: t \in \mathbb{R})$. Recall that $\mathcal{T}$ is the smallest topology on $C$ for which all projections are continuous.
Note that the trace $\mathcal{C}_S:= S \cap \mathcal{C}$ in $S$ is generated by the restrictions of $\pi_t$ to $S$.
Analogously, the subspace topology $\mathcal{T}_S = S \cap \mathcal{T}$ on $S$ is generated by these restrictions.
The corresponding statements hold for $S^\prime$ endowed with the trace $\mathcal{C}_{S'}=S' \cap \mathcal{C}$ and the subspace topology $\mathcal{T}_{S'}=S' \cap \mathcal{T}$. \\

\begin{proposition} \label{semi}
\begin{itemize}
\item[(i)]
$\sigma:S \rightarrow \mathbb{R}$ is $\mathcal{C}_S$-Borel measurable and $\mathcal{T}_S-$lower semicontinuous.
\item[(ii)]
$\tau:S' \rightarrow \mathbb{R}$ is $\mathcal{C}_{S'}$-Borel measurable and $\mathcal{T}_{S'}-$upper semicontinuous.
\end{itemize}
\end{proposition}

\begin{proof} For each $x \in \mathbb{R}$ we have that:
\begin{eqnarray*}
&&\{f \in S: \sigma(f) \le x \}\\
&=&\{f \in S: D^+f(x) \ge 0\} \hspace{2cm} \text{ by Theorem \ref{location}}\\
&=&\{f \in S: \inf_{x<t}\frac{f(t)-f(x)}{t-x} \ge 0\}\hspace{1cm} \text{by the first equality in (\ref{D+})}\\
&=&\{f \in S: \inf_{x<t \in \mathbb{Q}}\frac{f(t)-f(x)}{t-x} \ge 0\} \hspace{0.64cm} \text{by continuity of } f\\
&=& \bigcap_{x<t\in \mathbb{Q}} \{f \in S: \frac{f(t)-f(x)}{t-x} \ge 0\}\\
&=& \bigcap_{x<t\in \mathbb{Q}} \{f \in S: f(t)-f(x) \ge 0\}\\
&=& \bigcap_{x<t\in \mathbb{Q}} \big(S \cap (\pi_t-\pi_x)^{-1}([0,\infty))\big)\\
&=& S \cap \bigcap_{x<t\in \mathbb{Q}} (\pi_t-\pi_x)^{-1}([0,\infty))\\
\end{eqnarray*}
By construction of $\mathcal{C}$ every projection is $\mathcal{C}$-measurable, whence the differences $\pi_t-\pi_x$ are $\mathcal{C}$-measurable
as well and therefore $(\pi_t-\pi_x)^{-1}([0,\infty)) \in \mathcal{C}$ for all rationals $t>x$. Since $\mathcal{C}$ is closed under denumerable intersections we arrive at $\{f \in S: \sigma(f) \le x \} \in S \cap \mathcal{C} = \mathcal{C}_S$ for all $x \in \mathbb{R}$, which by Lemma 1.4 in Kallenberg \cite{Kall1} shows measurability of $\sigma$.

As to semicontinuity recall that by construction of $\mathcal{T}$ every projection $\pi_t$ is $\mathcal{T}$-continuous, whence the differences $\pi_t-\pi_x$ are $\mathcal{T}$-continuous as well and therefore $(\pi_t-\pi_x)^{-1}([0,\infty))$ are $\mathcal{T}-$closed for all $t>x$. Since $\mathcal{T}$ is closed under every kind of intersections we arrive at $\{f \in S: \sigma(f) \le x \}$ is $\mathcal{T}_S-$closed for all $x \in \mathbb{R}$, which shows
$\mathcal{T}_S-$semicontinuity of $\mathcal{\sigma}$.

The second part follows in the same way. Indeed, since $D^-f(x)=\sup_{t<x}\frac{f(t)-f(x)}{t-x}$ it follows analogously that
$$
 \{f \in S': \tau(f) \ge x \} = S' \cap \bigcap_{x<t\in \mathbb{Q}} (\pi_t-\pi_x)^{-1}([0,\infty)).
$$
\end{proof}

Next we give further equivalent characterizations of semi-continuity.
The first one is an immediate consequence of Proposition \ref{semi} and the definition of continuity (pre-images of open sets are open).\\

\begin{corollary} \label{ordertop}
Let $\mathcal{O}_<:=\{(-\infty,x): x \in \mathbb{R}\} \cup \{\emptyset, \mathbb{R}\}$ and $\mathcal{O}_>:=\{(x,\infty): x \in \mathbb{R}\}\cup \{\emptyset, \mathbb{R}\}$ be the left-order topology and the right-order topology.
Then:
\begin{itemize}
\item[(1)]
$\sigma:(S,\mathcal{T}_S) \rightarrow (\mathbb{R}, \mathcal{O}_>)$ is continuous.
\item[(2)]
$\tau:(S',\mathcal{T}_{S'}) \rightarrow (\mathbb{R}, \mathcal{O}_<)$ is continuous.
\end{itemize}
\end{corollary}

Sometimes it is advantageous to consider the restrictions of $\sigma$ and $\tau$ on subspaces.\\
\begin{remark} \label{U}
Let $\emptyset \neq U \subseteq S$ be endowed with $\mathcal{C}_U:=U \cap \mathcal{C}$ and $\mathcal{T}_U:= U \cap \mathcal{T}$.
Then $\sigma:U \rightarrow \mathbb{R}$ is $\mathcal{C}_U-$Borel
measurable and $\sigma:(U,\mathcal{T}_U) \rightarrow (\mathbb{R}, \mathcal{O}_>)$ is continuous.
Similarly, if $U \subseteq S'$, then $\tau:U \rightarrow \mathbb{R}$ is $\mathcal{C}_U-$Borel
measurable and $\tau:(U,\mathcal{T}_U) \rightarrow (\mathbb{R}, \mathcal{O}_<)$ is continuous.
\end{remark}

\vspace{0.5cm}
Corollary \ref{ordertop} in turn yields a second equivalent description of semi-continuity via net-convergence. For this purpose, let $(I,\le)$ be here and in the following a directed
set. Also recall the definition
\begin{equation} \label{Su}
 S_u :=\{f \in C: f \text{ has a unique minimizing point}\}.
\end{equation}

\begin{corollary} \label{netconv}
Assume that $(f_\alpha)_{\alpha \in I} \subseteq C$ converges pointwise to $f$ on $\mathbb{R}$. Then the following statements apply:
\begin{itemize}
\item[(1)]
If $(f_\alpha)_{\alpha \in I} \subseteq S$ and $f \in S$, then $\liminf_{\alpha} \sigma(f_\alpha) \ge \sigma(f)$.
\item[(2)]
If $(f_\alpha)_{\alpha \in I} \subseteq S'$ and $f \in S'$, then $\limsup_{\alpha} \tau(f_n) \le \tau(f)$.
\item[(3)]
If $(f_\alpha)_{\alpha \in I} \subseteq S \cap S'$ and $f \in S_u$, then $\lim_{\alpha} \sigma(f_\alpha)= \sigma(f)$ and $\lim_{\alpha} \tau(f_\alpha)= \tau(f)$. Note that $\sigma(f)=\tau(f)$, because $f \in S_u$. Thus the smallest minimizer and the largest minimizer converge to the same limit.
\end{itemize}
\end{corollary}

\begin{proof} By assumption, $f_\alpha \rightarrow f$ in $(C,\mathcal{T})$, which by the requirement in (1) is the same as $f_\alpha \rightarrow f$ in $(S,\mathcal{T}_S)$. According to Corollary \ref{ordertop}
$\sigma:(S,\mathcal{T}_S) \rightarrow (\mathbb{R}, \mathcal{O}_>)$ is continuous at every point $f \in S$.
Consequently $\sigma(f_\alpha) \rightarrow \sigma(f)$ in $(\mathbb{R}, \mathcal{O}_>)$. Now, a net $(y_\alpha)$ converges in $(\mathbb{R},\mathcal{O}_>)$ to
$y$
if and only if $\liminf_{\alpha} y_\alpha \ge y$, which gives (1).
In the same way one obtains (2) upon noticing that $y_\alpha \rightarrow y$ in $(\mathbb{R},\mathcal{O}_<)$
if and only if $\limsup_{\alpha} y_\alpha \le y$.
Finally, (3) follows from (1) and (2), because $\sigma \le \tau$ and therefore
$$
 \sigma(f) \le \liminf_{\alpha} \sigma(f_\alpha) \le \limsup_{\alpha} \sigma(f_\alpha) \le \limsup_{\alpha} \tau(f_\alpha) \le \tau(f) = \sigma(f).
$$
This shows that $\sigma(f_\alpha) \rightarrow \sigma(f)$. Similarly
$$
 \sigma(f) \le \liminf_{\alpha} \sigma(f_\alpha) \le \liminf_{\alpha} \tau(f_\alpha) \le \limsup_{\alpha} \tau(f_\alpha) \le \tau(f) = \sigma(f)
$$
resulting in $\tau(f_\alpha) \rightarrow \tau(f)$.
\end{proof}

\vspace{0.5cm}
Semi-continuity of $\sigma$ and $\tau$ as stated in Proposition \ref{semi} and its reformulations in Corollaries \ref{ordertop} and \ref{netconv} turn out to be
a very strong tool for proving so-called Argmin theorems in probability and statistics.\\

Occasionally it is stated in the literature that $\sigma$ or $\tau$ are actually continuous with respect to the natural topology $\mathcal{O}_n$ on $\mathbb{R}$. But the following examples
shows that this is not true.\\

\begin{example} Consider
\begin{equation*}
 f(t) = \left\{ \begin{array}{l@{\quad,\quad}l}
                 0 & |t| \le 1\\ |t|-1 & |t|>1
               \end{array} \right.
\end{equation*}
and for every $n \in \mathbb{N}$ let
\begin{equation*}
 f_n(t) = \left\{ \begin{array}{l@{\quad,\quad}l}
                 f(t) & t<0 \text{ or } t>1+\frac{1}{n}\\ \frac{1}{n+1} t & t \in [0,1+\frac{1}{n}].
               \end{array} \right.
\end{equation*}
Obviously, $f$ and $f_n, n \in \mathbb{N}$ are convex and $f_n$ converges at every point (actually uniformly on $\mathbb{R}$) to $f$.
However, $\tau(f_n)=0$ for all $n \in \mathbb{N}$, whereas $\tau(f)=1$ and consequently $\tau(f_n) \not\rightarrow \tau(f)$.
Thus $\tau$ is not continuous at $f$ and from $\sigma(f)=-\tau(f_{-})$ we infer that $\sigma$ is not continuous at $f_{-}$.
\end{example}

\vspace{0.5cm}
Note that the limit function in our example has no unique minimizing point. So let us consider the family $S_u$ in (\ref{Su}) of all functions $f$ with a unique minimizer. Clearly it holds that $S_u \subseteq S \cap S'$ and that the functionals $\sigma$ and $\tau$ coincide on $S_u$.
Therefore, from Remark \ref{U} we can infer that $\sigma$ is lower- and
upper-semicontinuous on the subspace $(S_u,\mathcal{T}_{S_u})$, whence $\sigma$ is continuous on $(S_u,\mathcal{T}_{S_u})$ with respect to the natural topology $\mathcal{O}_n$ on $\mathbb{R}$. We note this in the following\\

\begin{corollary} \label{cont} $\sigma = \tau$ on $S_u$ and
$$
 \sigma:(S_u,\mathcal{T}_{S_u}) \rightarrow (\mathbb{R},\mathcal{O}_n)
$$
is continuous.
\end{corollary}

\vspace{0.5cm}
Let $S^*:=S \cap S'=\{f:\mathbb{R}\rightarrow \mathbb{R}; f \text{ convex with }A(f) \text{ is a compact interval}\}$ and let $\xi:S^* \rightarrow \mathbb{R}$ be any measurable selection of $A$, i.e., $\xi(f) \in A(f)$ for every $f \in S^*$ and measurability refers to the trace $\mathcal{C}_{S^*}=S^* \cap \mathcal{C}$. Assume that $(f_\alpha) \subseteq S^*$ converges pointwise to $f\in S^*$. Since $\sigma(f_\alpha) \le \xi(f_\alpha) \le \tau(f_\alpha)$ for all $\alpha \in I$, it follows from the above Corollary \ref{netconv} (and the characterization of net-convergence in the order topologies, confer the proof of Corollary \ref{netconv}) that
$$
\xi(f_\alpha) \rightarrow \sigma(f) \;\text{ in } \; (\mathbb{R}, \mathcal{O}_>) \quad  \text{and } \; \xi(f_\alpha) \rightarrow \tau(f) \;\text{ in } \; (\mathbb{R},\mathcal{O}_<).
$$
If $f \in S_u$, then $\sigma(f)=\tau(f)$, whence $\xi(f_\alpha) \rightarrow \sigma(f)$ in $(\mathbb{R}, \mathcal{O}_>)$ and in $(\mathbb{R}, \mathcal{O}_<)$, which, as we know, is the same as $\xi(f_\alpha) \rightarrow \sigma(f)= \tau(f)$ in the natural topology $\mathcal{O}_n$. In particular, every measurable selection of $A$ is continuous on the subspace $S_u$ with limit $\sigma = \tau$.
We see here, and will see it another time later, that the class $S_u$ of convex functions with unique minimization point plays a special role.\\

\begin{lemma} \label{Sumb} $S_u \in \mathcal{C}_{S^*}$.
\end{lemma}

\begin{proof} 
Remark \ref{U} says that $\sigma:(S^*, \mathcal{C}_{S^*}) \rightarrow \mathbb{R}$ and $\tau:(S^*, \mathcal{C}_{S^*}) \rightarrow \mathbb{R}$  are Borel measurable. Infer from $S_u \subseteq S^*$ that $S_u=\{f \in S^*: \sigma(f)=\tau(f)\}=(\sigma-\tau)^{-1}(\{0\}) \in \mathcal{C}_{S^*}$.
\end{proof}

In addition to the topology $\mathcal{T}$ of pointwise convergence, let $C$ also be endowed with the topology $\mathcal{T}_{uc}$ of uniform convergence on compacta. It is well-known that $\mathcal{T} \subseteq \mathcal{T}_{uc}$, because uniform convergence on compacta implies pointwise convergence. From Theorem 10.8 of Rockafellar \cite{Rockafellar} we know that on $C$ the reverse is true. Notice that this is valid only for
sequences. Thus the identity $i:(C,\mathcal{T}) \rightarrow (C,\mathcal{T}_{uc})$ is sequentially continuous at every $f \in C$.
Unfortunately, in general topological spaces sequential continuity does not imply continuity, which in turn would give
$\mathcal{T}_{uc} = i^{-1}(\mathcal{T}_{uc}) \subseteq \mathcal{T}$ as desired. In fact the implication is true, if
the space is first countable, confer Theorem 7.1.3 in Singh \cite{Singh}. At this stage, however, we do not know whether first countability holds for $(C,\mathcal{T})$. So, we will prove continuity of $i:(C,\mathcal{T}) \rightarrow (C,\mathcal{T}_{uc})$ traditionally via net-convergence,
confer Theorem 4.2.6 in Singh \cite{Singh}. Theorem \ref{uniformnet} below on net-convergence is not only the key to success, but above all interesting in itself when compared with the sequential convergence occurring in Theorem 10.8 of Rockafellar \cite{Rockafellar}.
The proof of Theorem \ref{uniformnet} is based on the following inequality.\\

\begin{lemma} \label{inequality}
Let $D$ be a dense subset of $\mathbb{R}$. Then for every compact set $K \subseteq \mathbb{R}$ there exist a constant $C$ and points $d_1,\ldots,d_8 \in D$ such that for each convex function $f:\mathbb{R} \rightarrow \mathbb{R}$ it follows:
\begin{equation} \label{Lip}
 |f(s)-f(t)| \le C \sum_{i=1}^8 |f(d_i)| |s-t| \quad \forall \; s,t \in K.
\end{equation}
\end{lemma}

\begin{proof} First, find points $a$ and $b$ from $D$ such that $K \subseteq [a,b]$.
By Theorem 1.3.7 in Niculescu and Persson \cite{Niculescu} we have that:
$$
 |f(s)-f(t)| \le L |s-t| \quad \forall \; s,t \in [a,b],
$$
where $L= \max\{|D^+f(a)|,|D^-f(b)|\}$. Now,
\begin{equation} \label{a+}
 D^+f(a)=\inf_{x>a}\frac{f(x)-f(a)}{x-a} \le \frac{f(x)-f(a)}{x-a} \le (x-a)^{-1}(|f(x)|+|f(a)|) \quad \forall \; x>a
\end{equation}
and since $D^+f(a) \ge D^-f(a)= \sup_{y<a}\frac{f(y)-f(a)}{y-a}$ it follows that
\begin{equation} \label{a-}
 D^+f(a) \ge \frac{f(y)-f(a)}{y-a} \ge (y-a)^{-1}(|f(y)|+|f(a)|) \quad \forall \; y<a.
\end{equation}
Next, in (\ref{a+}) and in (\ref{a-}) we can choose the points $x$ and $y$ from $D$.
Put $C_1:=\max\{(x-a)^{-1},(a-y)^{-1}\} \in (0,\infty)$. Then $|D^+f(a)|\le C_1 (|f(x)|+|f(a)|+|f(y)|+|f(a)|)$.
Similarly one obtains:  $|D^-f(b)|\le C_2 (|f(u)|+|f(b)|+|f(v)|+|f(b)|)$ with positive constant $C_2$ and points
$u \in D, u<b$ and $v \in D, v >b$. Finally, if we put $C:=\max\{C_1,C_2\}$, then
$L \le C (|f(x)|+|f(a)|+|f(y)|+|f(a)|+|f(u)|+|f(b)|+|f(v)|+|f(b)|)$, which shows (\ref{Lip}).
\end{proof}

\vspace{0.5cm}
To state our next result recall that $(I,\le)$ is a directed set.\\

\begin{theorem} \label{uniformnet}
Let $D$ be dense in $\mathbb{R}$ and let $(f_\alpha)_{\alpha \in I}$ be a net in $C$, which converges pointwise on $D$ to a function $f$, that is $f_\alpha(t) \rightarrow f(t)$ for all $t \in D$. Then $f$ is convex and
$(f_\alpha)$ converges uniformly to $f$ on every compact subset of $\mathbb{R}$.
\end{theorem}

\begin{proof} Convexity of $f$ follows from the convexity of $f_\alpha$ by taking the limit.
As to the second assertion let $K \subseteq \mathbb{R}$ be compact. By Lemma \ref{inequality} there
are a constant $C$ and $d_1,\ldots,d_8 \in \mathbb{R}$ such that
\begin{equation} \label{equicont}
 |f_\alpha(s)-f_\alpha(t)| \le C \sum_{i=1}^8 |f_\alpha(d_i)| |s-t| \quad \forall \; s,t \in K \quad \forall \; \alpha \in I.
\end{equation}

By pointwise convergence we find for every $1 \le i \le 8$ an index $\alpha_i \in I$ and a constant $c_i \in \mathbb{R}$ such that $|f_\alpha(d_i)| \le c_i$ for all $\alpha \ge \alpha_i$. To $\alpha_1,\ldots,\alpha_8$ there exist a dominating index $\alpha^* \in I$ with $\alpha^* \ge \alpha_1,\ldots,\alpha_8$. It follows that $C \sum_{i=1}^8 |f_\alpha(d_i)| \le C \sum_{i=1}^8 c_i =:L \in [0,\infty)$ for all $\alpha \ge \alpha^*$, whence by
(\ref{equicont}) the family $\mathcal{F}:=\{f_\alpha: \alpha^* \le \alpha \in I\}$ is equicontinuous. Furthermore,
$$
 |f_\alpha(t)|\le |f_\alpha(t)-f_\alpha(d_1)|+|f_\alpha(d_1)|\le L|t-d_1|+c_1\quad \forall \alpha \ge \alpha^*\ge \alpha_1,
$$
whence $\mathcal{F}$ is pointwise bounded. Thus by the Arzel\`{a}-Ascoli theorem, confer, e.g., Heuser \cite{Heuser}, the family $\mathcal{F}$ is compact. Therefore,  if $(f_{\alpha'})$ is a subnet of $(f_\alpha)_{\alpha^* \le \alpha \in I}$, then there exists a further subnet $(f_{\alpha''})$ of $(f_{\alpha'})$, which converges to a function $g$ uniformly on $K$. In particular, $f_{\alpha''}(t) \rightarrow g(t)$ for all $t \in K$. But by the assumption of pointwise convergence we also know that $f_{\alpha''}(t) \rightarrow f(t)$ for all $t \in K$. Thus $g=f$ on $K$ and by the subnet-criterion it follows that $(f_\alpha)_{\alpha^* \le \alpha \in I}$ converges to $f$ uniformly on $K$, which a fortiori holds for the entire
net $(f_\alpha)_{\alpha \in I}$.
\end{proof}

Infer from the above Theorem \ref{uniformnet} that if a net $f_\alpha \rightarrow f$ in $(C,\mathcal{T})$, then $f_\alpha \rightarrow f$ in $(C,\mathcal{T}_{uc})$. Obviously, the reverse is true as well. As a consequence we obtain\\

\begin{corollary} \label{puc} The topology of pointwise convergence and the topology of uniform convergence on compacta coincide on $C$:
$$\mathcal{T} = \mathcal{T}_{uc}.$$
\end{corollary}

\begin{remark} \label{match}
Let $\mathcal{T}(D)$ be the topology of pointwise convergence on $D$. It is generated by the projections $\pi_t, t \in D$.
If $D$ is dense in $\mathbb{R}$, then Theorem \ref{uniformnet} actually yields that $\mathcal{T}(D)=\mathcal{T}_{uc}=\mathcal{T}$. So all the topologies match.
\end{remark}

\vspace{0.5cm}
The observation in Remark \ref{match} leads to the following variant of the semi-continuity.\\

\begin{corollary} \label{dense}
Let $D$ be dense in $\mathbb{R}$. If in Corollary \ref{netconv} the assumption is replaced by $(f_\alpha)_{\alpha \in I} \subseteq C$ converges pointwise to $f$ \textbf{on} $\textbf{D}$,
then all statements of Corollary \ref{netconv} remain valid.
\end{corollary}

\vspace{0.5cm}
Let $D=\{t_i: i \in \mathbb{N}\}$ be a countable and dense subset of $\mathbb{R}$. Introduce the special projection map
$H:C \rightarrow \mathbb{R}^\mathbb{N}$ by $H(f):=(f(t_i))_{i \in \mathbb{N}}$. Note that $H$ depends on $D$, but we suppress this in
our notation.
Equip $\mathbb{R}^\mathbb{N}$ with the product topology $\Pi$.
Denote the range $H(C)$ by $R$ and the relative topology $R \cap \Pi$ by $\mathcal{R}$. With the following result we will prove
a functional limit theorem for convex stochastic processes.\\

\begin{lemma} \label{projmap} The map $H$ is a bijection onto its range and its inverse $H^{-1}:(R,\mathcal{R}) \rightarrow (C,\mathcal{T})$
is continuous.
\end{lemma}

\begin{proof} If $H(f)=H(g)$, then $f=g$ on $D$ and by continuity and denseness of $D$
the equality holds on the entire real line. Thus $H$ is injective, and it is surjective by construction.

As to continuity of the inverse consider a sequence $(r_n)$ with
\begin{equation} \label{rn}
 r_n  \rightarrow r \; \text{ in } \; (R,\mathcal{R}).
\end{equation}
Since $(r_n) \subseteq R$, we find to each $n \in \mathbb{N}$ a function $f_n \in C$ such that
$r_n = H(f_n)=(f_n(t_i))_{i \in \mathbb{N}}$. For the same reason there is some $f \in C$ with
$r=H(f)=(f(t_i))_{i \in \mathbb{N}}$. Recall that convergence in $\Pi$ or in $\mathcal{R}$, respectively,
is the same as coordinate-wise convergence.
Thus the convergence in (\ref{rn}) means that $f_n(t_i) \rightarrow f(t_i)$ for all $i \in \mathbb{N}$.
Since $D$ lies dense in $\mathbb{R}$ we can apply Theorem 10.8 in Rockafellar \cite{Rockafellar} (or our Theorem \ref{uniformnet}) to infer
that actually $f_n(t) \rightarrow f(t)$ for every $t \in \mathbb{R}$. Now, by definition $f_n=H^{-1}(r_n)$ and
$f=H^{-1}(r)$, whence we arrive at $H^{-1}(r_n) \rightarrow H^{-1}(r)$ in $(C,\mathcal{T})$. Consequently, $H^{-1}$ is continuous.
\end{proof}

\section{Applications in probability and statistics}
Let $(\Omega,\mathcal{A})$ be a measurable space.
For a map $Z:\Omega \rightarrow C$ we write $Z(\omega,t):=Z(\omega)(t)$ for the value of the function $Z(\omega):\mathbb{R} \rightarrow \mathbb{R}$ (\emph{trajectory}) at point $t \in \mathbb{R}$. Very often it is more convenient to write $Z(t)$ instead of $Z(\omega,t)$ for this ambiguity in the notation explains in the context. Let $\mathcal{B}:=\mathcal{B}(\mathbb{R})$ denote the Borel-$\sigma$ algebra on $\mathbb{R}$. If $Z(\cdot,t):\Omega \rightarrow \mathbb{R}$ is $\mathcal{A}-\mathcal{B}$ measurable for each $t \in \mathbb{R}$, then $Z$ is called a \emph{convex stochastic process}. This is the same as
saying that $Z(t)$ is a real random variable for all $t \in \mathbb{R}$. If $Z(\omega) \in U$ for all $\omega \in \Omega$, where $U$ is a subset of $C$, we say that $Z$ is a \emph{convex stochastic process in} $U$ or for short \emph{a process in} $U$. In other words, all trajectories of $Z$ are $U-$valued.\\

Let $\mathcal{B}(C):=\sigma(\mathcal{T})$ be the Borel-$\sigma$ algebra pertaining to the topology of pointwise convergence.
By Corollary \ref{puc} it coincides with $\mathcal{B}_{uc}(C):= \sigma(\mathcal{T}_{uc})$. The following result yields a
convenient characterization of the Borel-$\sigma$ algebra. Recall that by definition $\mathcal{C}=\sigma(\pi_t: t \in \mathbb{R})$.\\

\begin{proposition} \label{char}
$$
 \mathcal{B}(C)=\mathcal{B}_{uc}(C)=\mathcal{C}.
$$
\end{proposition}

\begin{proof} Let $C^*:=\{f:\mathbb{R} \rightarrow \mathbb{R}; f \text{ continuous}\}$ be endowed with the topology $\mathcal{T}_{uc}^*$
of uniform convergence on compacta. One verifies easily that $\mathcal{T}_{uc}= C \cap \mathcal{T}_{uc}^*$. If $i:C \rightarrow C^*$ is the
natural injection into $C^*$, i.e., $i(f)=f$, then $\mathcal{T}_{uc}=i^{-1}(\mathcal{T}_{uc}^*)$, whence
$$
\sigma(\mathcal{T}_{uc})= \sigma(i^{-1}(\mathcal{T}_{uc}^*))=i^{-1}(\sigma(\mathcal{T}_{uc}^*))=C \cap \sigma(\mathcal{T}_{uc}^*),
$$
where the second equality is ensured by Lemma 1.2.5 in G\"{a}nssler and Stute \cite{Stute}.
By Lemma A5.1 in Kallenberg \cite{Kall2} we have that $\sigma(\mathcal{T}_{uc}^*) = \sigma(\pi_t^*: t \in \mathbb{R})$, where $\pi_t^*:C^* \rightarrow \mathbb{R}$ is the projection on $C^*$. But $C \cap \sigma(\pi_t^*: t \in \mathbb{R})=\sigma(\pi_t: t \in \mathbb{R})$, which gives the desired result.
\end{proof}

If $Z$ is a process in $U \subseteq C$ it can be regarded as a map $Z:\Omega \rightarrow U$ with Borel$-\sigma$ algebra $\mathcal{B}(U)=\sigma(\mathcal{T}_U)=U \cap \mathcal{B}(C)= U \cap \mathcal{C} = \mathcal{C}_U$ by Proposition \ref{char}.\\

\begin{lemma} \label{mb}
Assume that $Z$ is a convex stochastic process. Then
$Z$ is $\mathcal{A}-\mathcal{B}(C)$ measurable. If actually $Z$ is a process in $U \subseteq C$, then $Z$ is $\mathcal{A}-\mathcal{B}(U)$ measurable.
\end{lemma}

\begin{proof} The first assertion follows from Proposition \ref{char}, which enables us to apply Proposition 1.2.11
in G\"{a}nssler and Stute \cite{Stute}. The second assertion follows from $\mathcal{B}(U)=U \cap \mathcal{B}(C)$ in combination with
the first assertion.
\end{proof}

Recall that $S^*=S \cap S'$ and $\xi:(S^*,\mathcal{C}_{S^*}) \rightarrow (\mathbb{R},\mathcal{B})$ denotes any measurable selection of $A$.\\

\begin{corollary} \label{mbargmin} If $Z$ is a convex stochastic process in $S$, in $S'$ or in $S^*$, then
$\sigma(Z), \; \tau(Z)$ or $\xi(Z)$, respectively, are real random variables.
\end{corollary}

\begin{proof} Lemma \ref{mb} says that $Z:(\Omega,\mathcal{A}) \rightarrow (S,\mathcal{B}(S))$ is measurable.
But $\mathcal{B}(S)=S \cap \mathcal{B}(C)=S \cap \mathcal{C}= \mathcal{C}_S$, where the second equality holds by Proposition \ref{char}.
From Proposition \ref{semi} we know that $\sigma:(S,\mathcal{C}_S) \rightarrow (\mathbb{R},\mathcal{B})$ is measurable, whence $\sigma(Z)=\sigma \circ Z$ is measurable as composition of measurable maps. Replacing $S$ through $S'$ or $S^*$ gives measurability of $\tau(Z)$ or $\xi(Z)$, respectively.
\end{proof}

The concept of \emph{convergence in distribution} is well-known for random variables with values in a metric space.
A classical reference here is the book of Billingsley \cite{Bill}. In contrast, the extension of the concept
from metric spaces to topological spaces seems less known. It goes back to G\"{a}nssler and Stute \cite{Stute}, who in turn modify the ideas of Tops{\o}e \cite{Top}. Let $Z$ and $Z_\alpha, \alpha \in I$, be random variables defined on a probability space $(\Omega, \mathcal{A}, \mathbb{P})$ with values in some topological space $(X,\mathcal{O})$, that is $Z:\Omega \rightarrow X$ and $Z_\alpha:\Omega \rightarrow X$ are $\mathcal{A}-\mathcal{B}(X)$
measurable, where $\mathcal{B}(X):=\sigma(\mathcal{O})$ denotes the Borel-$\sigma$ algebra.
Then the net $(Z_\alpha)_{\alpha \in I}$ \emph{converges in distribution to} $Z$ \emph{in} $(X,\mathcal{O})$, if
\begin{equation} \label{open}
 \liminf_\alpha \mathbb{P}(Z_\alpha \in O) \ge \mathbb{P}(Z \in O) \quad \forall \; O \in \mathcal{O}.
\end{equation}
This is denoted by $Z_\alpha \stackrel{\mathcal{D}}{\rightarrow} Z$ in $(X,\mathcal{O})$ and by complementation is equivalent to
\begin{equation} \label{closed}
 \limsup_\alpha \mathbb{P}(Z_\alpha \in F) \le \mathbb{P}(Z \in F) \quad \forall \; F \in \mathcal{F},
\end{equation}
where $\mathcal{F}$ is the family of all closed sets in $(X,\mathcal{O})$.\\

The following result plays an important role in what follows. For this reason we like to state it here.
The proof is comparatively simple and can be found in G\"{a}nssler and Stute \cite{Stute}, p.345.\\

\begin{theorem} (\textbf{Continuous Mapping}) Let $(X,\mathcal{O})$ and $(E,\mathcal{G})$ be topological spaces,
$h:X \rightarrow E$ \;be  $\mathcal{B}(X)-\mathcal{B}(E)$ measurable and $D_h:=\{x \in X: h \text{ is discontinuous at } x \}$.
Suppose $Z$ and $Z_\alpha, \alpha \in I$, are random variables over $(\Omega, \mathcal{A},\mathbb{P})$ with values in $(X,\mathcal{O})$,
where $\mathbb{P}^*(Z \in D_h)=0$ with $\mathbb{P}^*$ the outer measure of $\mathbb{P}$. Then
$Z_\alpha \stackrel{\mathcal{D}}{\rightarrow} Z$ in $(X,\mathcal{O})$ entails $h(Z_\alpha) \stackrel{\mathcal{D}}{\rightarrow} h(Z)$ in $(E,\mathcal{G})$.
\end{theorem}

\vspace{0.5cm}
Let $(Y,\mathcal{O}_Y)$ with $Y \subseteq X$ and $\mathcal{O}_Y:=Y \cap \mathcal{O}$ be a subspace of $(X,\mathcal{O})$. Notice that $\mathcal{B}(Y)=
\sigma(\mathcal{O}_Y)= Y \cap \mathcal{B}(X)$. So, a map $Z:\Omega \rightarrow X $ with range contained in $Y$ is
$\mathcal{A}-\mathcal{B}(Y)$ measurable (considered as a map into $Y$) if and only if it is $\mathcal{A}-\mathcal{B}(X)$ measurable.
Suppose $Z$ and $Z_\alpha$ for all $\alpha \in I$ are random variables with values in the subspace. Then $Z_\alpha \stackrel{\mathcal{D}}{\rightarrow} Z$ in $(Y,\mathcal{O}_Y)$ is equivalent to
$Z_\alpha \stackrel{\mathcal{D}}{\rightarrow} Z$ in $(X,\mathcal{O})$. Indeed, the natural injection $i:(Y,\mathcal{O}_Y) \rightarrow (X,\mathcal{O})$ given by $i(x)=x$ is continuous. Thus the \emph{Contiuous Mapping Theorem (CMT)} shows sufficiency.  The necessity follows from (\ref{open}) upon noticing that $\{Z_\alpha \in Y \cap O\}=\{Z_\alpha \in O\}$ and $\{Z \in Y \cap O\}=\{Z \in O\}$ for all $O \in \mathcal{O}$. We call this equivalence the \emph{Subspace-lemma}.

For further properties including the \emph{Portmanteau-Theorem} we refer to chapter 8.4 in G\"{a}nssler and Stute \cite{Stute}.\\

Recall the left- and right-order topologies $\mathcal{O}_<$ and $\mathcal{O}_>$, which are not metrizible. If a net $(x_\alpha)_{\alpha \in I}$ converges in $(\mathbb{R}, \mathcal{O}_<)$ and in $(\mathbb{R}, \mathcal{O}_>)$, then it converges in the natural topology $\mathcal{O}_n$, and vice versa.
The following example shows that there is a counterpart for distributional convergence.\\

\begin{example} \label{disordertop}
\begin{equation} \label{left}
 Z_\alpha \stackrel{\mathcal{D}}{\rightarrow} Z \; \text{ in } \; (\mathbb{R},\mathcal{O}_>) \quad \Longleftrightarrow \quad  \liminf_\alpha \mathbb{P}(Z_\alpha >x) \ge \mathbb{P}(Z >x) \quad \forall \; x \in \mathcal{R}.
\end{equation}
\begin{equation} \label{right}
 Z_\alpha \stackrel{\mathcal{D}}{\rightarrow} Z \; \text{ in } \; (\mathbb{R},\mathcal{O}_<) \quad \Longleftrightarrow \quad \liminf_\alpha \mathbb{P}(Z_\alpha <x) \ge \mathbb{P}(Z <x) \quad \forall \; x \in \mathcal{R}.
\end{equation}
\begin{equation} \label{leftright}
 Z_\alpha \stackrel{\mathcal{D}}{\rightarrow} Z \; \text{ in } \; (\mathbb{R},\mathcal{O}_>) \text{ and in } \; (\mathbb{R},\mathcal{O}_<) \quad \Longleftrightarrow \quad  \; Z_\alpha \stackrel{\mathcal{D}}{\rightarrow} Z \; \text{ in } \; (\mathbb{R},\mathcal{O}_n)
\end{equation}

\vspace{0.5cm}
Here, (\ref{left}) and (\ref{right}) are immediate consequences of the definitions. In (\ref{leftright}) the sufficiency of the right side follows from $\mathcal{O}_n \supseteq \mathcal{O}_<$ and $\mathcal{O}_n \supseteq \mathcal{O}_>$. To see necessity let $x \in \mathbb{R}$. Then we obtain:

\begin{eqnarray*}
 \mathbb{P}(Z<x) &\le& \liminf_\alpha \mathbb{P}(Z_\alpha <x) \quad \hspace{2cm} \text{ by } (\ref{right})\\
                 &\le& \limsup_\alpha \mathbb{P}(Z_\alpha <x) \\
                 &\le& \limsup_\alpha \mathbb{P}(Z_\alpha \le x) \le \mathbb{P}( Z \le x) \quad \text{ by } (\ref{left}) \text{ and complementation}.
\end{eqnarray*}
Thus, if the distribution function of $Z$ is continuous at $x$, i.e., $\mathbb{P}(Z<x)=\mathbb{P}(Z \le x)$ it follows that
$\lim_\mathbb{\alpha} \mathbb{P}(Z_\alpha \le x) = \mathbb{P}( Z \le x)$ as required.\\

Deduce from (\ref{left}): If $Z_\alpha \le Z_\alpha^* \; \mathbb{P}-$almost surely (a.s.) for all $\alpha \ge \alpha_0 \in I$ or if $Z \ge Z^*$ a.s., then  $Z_\alpha \stackrel{\mathcal{D}}{\rightarrow} Z \; \text{ in } \; (\mathbb{R},\mathcal{O}_>)$ entails  $Z_\alpha^* \stackrel{\mathcal{D}}{\rightarrow} Z^* \; \text{ in } \; (\mathbb{R},\mathcal{O}_>)$. (property 1)

Deduce from (\ref{right}): If $Z_\alpha \ge Z_\alpha^*$ (a.s.) for all $\alpha \ge \alpha_0 \in I$ or if $Z \le Z^*$ a.s., then  $Z_\alpha \stackrel{\mathcal{D}}{\rightarrow} Z \; \text{ in } \; (\mathbb{R},\mathcal{O}_<)$ entails  $Z_\alpha^* \stackrel{\mathcal{D}}{\rightarrow} Z^* \; \text{ in } \; (\mathbb{R},\mathcal{O}_<)$. (property 2)\\

In particular, this shows that the limit variables are not unique.
\end{example}

\vspace{0.5cm}
Our next result characterizes distributional convergence for random variables with values in the function space $(C,\mathcal{T})$.
Below the euclidian space $\mathbb{R}^k$ is endowed with the product-topology $\mathcal{O}_n^k$.\\

\begin{proposition} (\textbf{Functional limits}) \label{fidis} Let $Z$ and $Z_\alpha, \alpha \in I$, be convex stochastic processes. Then they are random variables in $(C,\mathcal{T})$ and the following statements (\ref{convinC}) and (\ref{convfidis}) are equivalent:
\begin{equation} \label{convinC}
 Z_\alpha \stackrel{\mathcal{D}}{\rightarrow} Z \; \text{ in } \; (C,\mathcal{T})
\end{equation}

\begin{equation} \label{convfidis}
 (Z_\alpha(t_1),\ldots,Z_\alpha(t_k)) \stackrel{\mathcal{D}}{\rightarrow} (Z(t_1),\ldots,Z(t_k)) \; \text{ in } (\mathbb{R}^k, \mathcal{O}_n^k)
\end{equation}
for every $k \in \mathbb{N}$ and for each collection of points $t_1,\ldots, t_k \in D$, where $D$ is a countable and dense subset of $\mathbb{R}$.
\end{proposition}

\begin{proof} The first assertion holds by Lemma \ref{mb}. Assume (\ref{convinC}) holds. By definition of $\mathcal{T}$ every projection $\pi_t:(C,\mathcal{T}) \rightarrow (\mathbb{R}, \mathcal{O}_n)$ is continuous, whence the product map $\pi:=(\pi_{t_1},\ldots,\pi_{t_k}):(C,\mathcal{T}) \rightarrow (\mathbb{R}^k, \mathcal{O}_n^k)$ is continuous as well. Since $(Z_\alpha(t_1),\ldots,Z_\alpha(t_k))= \pi(Z_\alpha)$ and $(Z(t_1),\ldots,Z(t_k))= \pi(Z)$,
an application of the CMT yields (\ref{convfidis}).

For the converse first note that by countability we have that $D=\{t_1,t_2,\ldots\}$. Recall the projection map $H$ given in Lemma \ref{projmap}.
It follows from Example 2.6 in Billingsley \cite{Bill2} that (\ref{convfidis}) entails that $H(Z_\alpha) \stackrel{\mathcal{D}}{\rightarrow} H(Z)$ in
$(\mathbb{R}^\mathbb{N}, \Pi)$. Now, the Subspace-lemma says that $H(Z_\alpha) \stackrel{\mathcal{D}}{\rightarrow} H(Z)$ in
$(R, \mathcal{R})$. By Lemma \ref{projmap} the inverse $H^{-1}:(R, \mathcal{R}) \rightarrow (C,\mathcal{T})$ is continuous, so that another application of the CMT yields (\ref{convinC}).
\end{proof}

The second statement (\ref{convfidis}) is known as \emph{convergence of the finite dimensional distributions (on D)} and denoted by
$Z_\alpha \stackrel{fd}{\longrightarrow}_D Z$ (in short: \emph{convergence of the fidis}).\\

We are now in the position to formulate several so-called \emph{Argmin-Theorems} for convex stochastic processes.\\

\begin{theorem} \label{notunique} Let $D$ be countable and dense in $\mathbb{R}$. Consider convex stochastic processes $Z$ and
$Z_\alpha, \alpha \in I$, in $U$. Suppose that $Z_\alpha \stackrel{fd}{\longrightarrow}_D Z$.
Then the following statements hold:
\begin{itemize}
\item[(1)]
If $U=S$, then $\sigma(Z_\alpha) \stackrel{\mathcal{D}}{\longrightarrow} \sigma(Z)$ in $(\mathbb{R},\mathcal{O}_>)$.\\
If in addition $Z$ is a process in $S^*$ with $Z \in S_u$ a.s., then $\sigma(Z_\alpha) \stackrel{\mathcal{D}}{\longrightarrow} \sigma(Z)$ in $(\mathbb{R},\mathcal{O}_n)$.
\item[(2)]
If $U=S'$, then $\tau(Z_\alpha) \stackrel{\mathcal{D}}{\longrightarrow} \tau(Z)$ in $(\mathbb{R},\mathcal{O}_<)$.\\
If in addition $Z$ is a process in $S^*$ with $Z \in S_u$ a.s., then $\tau(Z_\alpha) \stackrel{\mathcal{D}}{\longrightarrow} \tau(Z)$ in $(\mathbb{R},\mathcal{O}_n)$.
\item[(3)]
If $U=S^*$, then $\xi(Z_\alpha) \stackrel{\mathcal{D}}{\longrightarrow} \sigma(Z)$ in $(\mathbb{R},\mathcal{O}_>)$ and $\xi(Z_\alpha) \stackrel{\mathcal{D}}{\longrightarrow} \tau(Z)$ in $(\mathbb{R},\mathcal{O}_<)$.\\
If in addition $\sigma(Z) \stackrel{\mathcal{D}}{=} \tau(Z)$, then $\xi(Z_\alpha) \stackrel{\mathcal{D}}{\longrightarrow} \sigma(Z)$ in $(\mathbb{R},\mathcal{O}_n)$.
\end{itemize}
\end{theorem}

\begin{proof} First notice that by Corollary \ref{mbargmin} all involved maps are real random variables.
By Proposition \ref{fidis} $Z_\alpha \stackrel{\mathcal{D}}{\longrightarrow} Z$ in $(C,\mathcal{T})$, whence by the
Subspace-lemma
\begin{equation} \label{convinS}
Z_\alpha \stackrel{\mathcal{D}}{\longrightarrow} Z \text{ in } (S,\mathcal{T}_S).
\end{equation}
From Corollary \ref{ordertop} we know
that $\sigma:(S,\mathcal{T}_S) \rightarrow (\mathbb{R},\mathcal{O}_>)$ is continuous and consequently the CMT yields the first convergence in (1).

By Lemma \ref{Sumb} $S_u \in \mathcal{C}_{S^*}$. According to Lemma \ref{mb} $Z$ is $\mathcal{A}-\mathcal{C}_{S^*}$ measurable
upon noticing that $\mathcal{B}(S^*)= S^* \cap \mathcal{B}(C) = S^* \cap \mathcal{C}=\mathcal{C}_{S^*}$ by Proposition \ref{char}.
Thus $\{Z \in S_u\} \in \mathcal{A}$, whence by Corollary \ref{cont} it follows that
$0 \le \mathbb{P}^*(Z \in D_\sigma) \le \mathbb{P}^*(Z \notin S_u) = \mathbb{P}(Z \notin S_u)=0$.
Finally, by Propositions \ref{semi} and \ref{char} $\sigma:(S,\mathcal{T}_S) \rightarrow (\mathbb{R},\mathcal{O}_n)$ is $\mathcal{B}(S)-\mathcal{B}(\mathbb{R})$ measurable. Consequently by (\ref{convinS}) another application of the CMT gives the second convergence in (1).

The second part (2) follows exactly the same way. Regarding part (3) it should be noted that $S^* \subseteq S$ and $S^* \subseteq S'$. Therefore $\sigma(Z_\alpha) \stackrel{\mathcal{D}}{\longrightarrow} \sigma(Z)$ in $(\mathbb{R},\mathcal{O}_>)$ by (1) and
$\tau(Z_\alpha) \stackrel{\mathcal{D}}{\longrightarrow} \tau(Z)$ in $(\mathbb{R},\mathcal{O}_<)$ by (2). This shows the first assertions in
part (3), because $\sigma(Z_\alpha) \le \xi(Z_\alpha) \le \tau(Z_\alpha)$ for all $\alpha \in I$ and so we can use property 1
and property 2 from Example \ref{disordertop}. Finally, the second assertion in (3) follows from (\ref{leftright}) in Example \ref{disordertop}.
\end{proof}

\vspace{0.5cm}
Our Argmin-Theorems (1)-(3) involve two types of uniqueness assumptions. In (1) and (2) the requirement $Z \in S_u$ a.s. is the same
as $\sigma(Z)=\tau(Z)$ a.s., which in turn implies $\sigma(Z) \stackrel{\mathcal{D}}{=} \tau(Z)$ as in part (3). However, if
$\sigma(Z)$ and $\tau(Z)$ are $\mathbb{P}-$integrable, then the reverse implication holds. Indeed, in this case by linearity $\mathbb{E}[\tau(Z)-\sigma(Z)]=0$. But since the integrand is non-negative, it must be equal to zero a.s.

A comparison shows that on the one hand in (3) the uniqueness condition is weaker than in (1) or in (2), on the other hand in (3) the demands on the stochastic processes are more stringent.\\

In the following we give some interesting equivalent characterizations for almost sure uniqueness of the minimizing point.\\

\begin{proposition} \label{unique} Suppose $Z$ is a convex stochastic process in $S^*$.
Then the following three statements are equivalent:
\begin{itemize}
\item[(1)]
$\sigma(Z)=\tau(Z)$ a.s.

\item[(2)]
$x \notin A(Z)$  a.s. for Lebesgue-almost every $x \in \mathbb{R}$.

\item[(3)]
$0 \notin [D^-Z(x),D^+Z(x)]$ a.s. for Lebesgue-almost every $x \in \mathbb{R}$.
\end{itemize}
\end{proposition}

\begin{proof} We write briefly $\sigma=\sigma(Z)$ and $\tau=\tau(Z)$. Moreover, let $\lambda$  denote the Lebegue-measure on $\mathbb{R}$.  Note that $\tau-\sigma \ge 0$, whence by Corollary \ref{mbargmin} the expectation exists and
is equal to:
\begin{eqnarray*}
\mathbb{E}[\tau-\sigma]&=&\mathbb{E}[\lambda([\sigma,\tau])] = \int_\Omega \int_\mathbb{R}  1_{\{\sigma \le x, \tau \ge x\}} \lambda(dx)
\mathbb{P}(d\omega)\\
&=& \int_\Omega \int_\mathbb{R}  1_{\{D^-Z(x) \le 0 \le D^+Z(x)\}} \lambda(dx)\mathbb{P}(d\omega) \quad \text{by Theorem \ref{location}}\\
&=&\int_\mathbb{R} \mathbb{P}(D^-Z(x) \le 0 \le D^+Z(x)) \lambda(dx) \hspace{0,75cm} \text{by Fubini}\\
&=&\int_\mathbb{R} \mathbb{P}(x \in A(Z)) \lambda(dx) \hspace{3cm} \text{by } (\ref{minset})
\end{eqnarray*}
Observe that all occurring integrands are non-negative. Thus the equivalence can be deduced from a well-known result from integration theory, confer, e.g., Lemma 1.15, p. 304 in Dshalalow \cite{Dshal}.
\end{proof}

\vspace{0.5cm}
If $Z$ in the above proposition is in addition differentiable, then $\sigma(Z)=\tau(Z)$ a.s. if and only if $Z'(x) \neq 0$ a.s. for $\lambda-$ every $x \in \mathbb{R}$.\\

Next, we establish an Argmin-Theorem for almost sure convergence. Just as in Theorem \ref{notunique}, there are also convergence statements here for the non-unique case. This extends results known so far, confer Theorem 7.77 in Liese and Mieschke \cite{Liese}, which in turn rely on the unpublished preprint of Hjort and Pollard \cite{Hjort}.\\


\begin{theorem} \label{as} Let $Z$ and $Z_\alpha, \alpha \in I$, be convex stochastic processes in $U$ defined on a complete
probability space $(\Omega,\mathcal{A},\mathbb{P})$.
Assume that $Z_\alpha(t) \rightarrow Z(t)$ a.s. for every $t \in D$ with $D$ a countable and dense subset of $\mathbb{R}$.
\begin{itemize}
\item[(1)]
If $U=S$, then $\liminf_{\alpha} \sigma(Z_\alpha) \ge \sigma(Z)$ a.s.
\item[(2)]
If $U=S'$, then $\limsup_{\alpha} \tau(Z_\alpha) \le \tau(Z)$ a.s.
\item[(3)]
If $U=S^*$ and $Z \in S_u$ a.s., then $\lim_{\alpha} \sigma(Z_\alpha)= \sigma(Z)=\tau(Z)$ a.s. and $\lim_{\alpha} \tau(Z_\alpha)= \tau(Z)=\sigma (Z)$ a.s.
\item[(4)]
If $U=S^*$, then $\liminf_{\alpha} \xi(Z_\alpha) \ge \sigma(Z)$ a.s. and $\limsup_{\alpha} \xi(Z_\alpha) \le \tau(Z)$ a.s.
If in addition $Z \in S_u$ a.s., then $\lim_{\alpha} \xi(Z_\alpha) = \sigma(Z)= \tau(Z)$ a.s.
\end{itemize}
\end{theorem}

\begin{proof} Put $\Omega_t:=\{Z_\alpha(t) \rightarrow Z(t)\}, t \in D$. Then $\Omega_t \in \mathcal{A}$ and
$\mathbb{P}(\Omega_t)=1$ by assumption and by completeness. Thus
$\Omega_0 := \bigcap_{t \in D} \Omega_t \in \mathcal{A}$ and $\mathbb{P}(\Omega_0)=1$, because $D$ is countable.
Corollary \ref{dense}  yields that $\Omega_0 \subseteq \{\liminf_{\alpha} \sigma(Z_\alpha) \ge \sigma(Z)\} =:\Omega_1$.
It follows from completeness that $\Omega_1 \in \mathcal{A}$, and by monotonicity of $\mathbb{P}$ we arrive at $\mathbb{P}(\Omega_1)=1$, which shows (1). In the same way one obtains (2). Furthermore, Corollary \ref{dense} ensures that $\Omega_0 \cap \{Z \in S_u\}$ is a subset of $\{\lim_{\alpha} \sigma(Z_\alpha)= \sigma(Z)\}$ and of $\{\lim_{\alpha} \tau(Z_\alpha)= \tau(Z)\}$.
Again by completeness this gives (3). Since $S^* \subseteq S$ and $S^* \subseteq S'$ as well as $\sigma(Z_\alpha) \le \xi(Z_\alpha) \le \tau(Z_\alpha)$ the first part of (4) follows from (1) and (2). Finally, the second part of (4) follows from (3) in combination with the sandwich theorem.
\end{proof}

\begin{remark} \label{drop}
If $(Z_\alpha)_{\alpha \in \mathbb{N}}$ in Theorem \ref{as} in fact is a sequence, then we can drop the completeness
assumption about the underlying probability space $(\Omega,\mathcal{A},\mathbb{P})$. The reason for this is Corollary \ref{mbargmin}, which
guarantees that the sets $\Omega_t, \Omega_0$ and $\Omega_1$ are elements of $\mathcal{A}$.
\end{remark}

\vspace{0.5cm}
The Argmin-Theorem for convergence in probability takes the form as described below.
The reader recognises that it is only formulated for sequences and not more generally for nets. This is because we carry out the proof via the subsequence criterion. To the best of our knowledge, there is no counterpart for nets.\\

\begin{theorem} \label{convinprobab} Let $Z_n, n \in \mathbb{N}$, be convex stochastic processes in $U$ and let $Z \in S^*$ have an almost surely unique minimizer.
Suppose $Z_n(t) \stackrel{\mathbb{P}}{\rightarrow} Z(t)$ for each $t \in D$, where $D$ is a countable and dense subset of $\mathbb{R}$.
\begin{itemize}
\item[(1)]
If $U=S$, then $\sigma(Z_n) \stackrel{\mathbb{P}}{\rightarrow} \sigma(Z)\stackrel{a.s.}{=}\tau(Z)$.
\item[(2)]
If $U=S'$, then $\tau(Z_n) \stackrel{\mathbb{P}}{\rightarrow} \tau(Z)\stackrel{a.s.}{=}\sigma(Z)$.
\item[(3)]
If $U=S^*$, then $\xi(Z_n) \stackrel{\mathbb{P}}{\rightarrow} \sigma(Z)\stackrel{a.s.}{=}\tau(Z)$
\end{itemize}
\end{theorem}

\begin{proof} We will use the subsequence criterion, confer, e.g., Lemma 5.2 in Kallenberg \cite{Kall1}.
So, let $(n_{0k})_{k \in \mathbb{N}}$ be a subsequence of $\mathbb{N}$. As a countable set $D$ has the form
$D=\{t_1,t_2,\ldots\}$. From $Z_n(t_1) \stackrel{\mathbb{P}}{\rightarrow} Z(t_1)$ it follows with the subsequence criterion that there exists
a subsequence $(n_{1i})_{ i \in \mathbb{N}}$ of  $(n_{0k})_{k \in \mathbb{N}}$ such that on a set, $\Omega_1$ say, with probability one
we have that $Z_{n_{1i}}(t_1) \rightarrow Z(t_1), i \rightarrow \infty$. Another application of the subsequence criterion ensures that there
exists a subsequence $(n_{2i})_{i \in \mathbb{N}}$ of $(n_{1i})_{ i \in \mathbb{N}}$ such $Z_{n_{2i}}(t_2) \rightarrow Z(t_2), i \rightarrow \infty$,
on a set, $\Omega_2$ say, which has probability one. Continuing in this way we find for every $j \ge 1$ a subsequence $(n_{ji})_{i \in \mathbb{N}}$ of $(n_{j-1,i})_{i \in \mathbb{N}}$ and a set $\Omega_j$ with $\mathbb{P}(\Omega_j)=1$ such that $Z_{n_{ji}}(t_j) \rightarrow Z(t_j), i \rightarrow \infty$, on $\Omega_j$. Set $\Omega_0:=\bigcap_{j \ge 1} \Omega_j$. Observe that $\mathbb{P}(\Omega_0)=1$. The ''diagonal'' $(n_{ii})_{i \in \mathbb{N}}$ is a subsequence
of the given sequence $(n_{0k})_{k \in \mathbb{N}}$. Consider an arbitrary index $j \ge 1$. Apart from the first $j-1$ terms the sequence $(n_{ii})_{i \in \mathbb{N}}$ is a subsequence of $(n_{ji})_{i \in \mathbb{N}}$. But on $\Omega_0$, along that sequence we have convergence at point $t_j$, whence in particular $Z_{n_{ii}}(t_j) \rightarrow Z(t_j), i \rightarrow \infty$.
Thus we arrive at $Z_{n_{ii}}(t) \rightarrow Z(t), i \rightarrow \infty$, for all $t \in D$ a.s. Especially, the sequence $(Z_{n_{ii}})_{i \in \mathbb{N}}$ fulfils the requirements of Theorem \ref{as}, which yields that $\sigma(Z_{n_{ii}}) \rightarrow \sigma(Z)$ a.s. or that $\tau(Z_{n_{ii}}) \rightarrow \tau(Z)$ a.s., respectively, and another application of the subsequence criterion gives (1) and (2). The assertion (3) follows
from Theorem \ref{as} part (4).
\end{proof}

\vspace{0.5cm}
\textbf{Concluding remarks}

Argmin theorems for convex processes (and distributional convergence) are known since Davis, Knight and Liu (1992) \cite{Davis}, Hjort and Pollard (1993) \cite{Hjort}, Geyer (1996) \cite{Geyer} or Knight (2002) \cite{Knight}.
Here the processes can even be defined on $\mathbb{R}^d$ and not merely on $\mathbb{R}$ as in this paper.
Apart from that, however, there are two notable differences.
First, the publications mentioned above only consider sequences, whereas we more generally allow for nets of processes.
Secondly, a further main difference is that we also provide results when the limit process does not have a unique minimizing point.
Incidentally, both points also apply to the Argmin theorems for almost sure convergence.
The idea behind this is to understand semi-continuity of $\sigma$ or
$\tau$ as continuity with respect to the order topologies in place of the natural topology on $\mathbb{R}$. After that the Continuous Mapping Theorem does the rest.
Another way of looking at our approach is this: As long as a unique minimizer of the limit process exists, it is a natural candidate for the limit variable. If this is not the case, we do not search for new candidates, but simply make the topology on $\mathbb{R}$ smaller. This also differs from Ferger's \cite{Ferger} innovative approach, which retains the natural topology but more generally allows Choquet-capacities to play the part of limit ''distributions''. The applicability of the Argmin theorem for convex processes lies in the fact that here, in contrast to such processes in larger function spaces, the only prerequisite is that of convergence of the finite dimensional distributions, confer Ferger \cite{Ferger3} for an application in $M-$estimation.
For example let $C^*=\{f:\mathbb{R}\rightarrow \mathbb{R}; f \text{ is continuous}\} \supset C$ be endowed with the topology $\mathcal{T}_{uc}^*$
of uniform convergence on compacta. Here, even if you only consider sequences, one needs not only a functional limit theorem $Z_n \stackrel{\mathcal{D}}{\mathcal{\rightarrow}} Z$ in $(C^*,\mathcal{T}_{uc}^*)$, but also stochastic boundedness of the $\xi(Z_n)$:
\begin{equation} \label{bound}
 \lim_{d \rightarrow \infty} \limsup_{n \rightarrow \infty} \mathbb{P}(|\xi(Z_n)|>d)=0,
\end{equation}
see Ibragimov and Hasminski (1981) \cite{Ibragimov}, van der Vaart and Wellner (1996) \cite{Vaart} or Ferger (2015) \cite{Ferger2}.
For the proof of the functional limit theorem alone, besides convergence of the fidis also tightness of the sequence $(Z_n)$ is required, which is usually done through maximal inequalities. Not to forget the proof of (\ref{bound}), usually by upper estimates for the tail probabilities.
This means that the programme that has to be worked through is much more extensive and demanding than in the convex case.

There are two answers to the question why this is so. Firstly, by Proposition \ref{fidis} it is already true under the sole assumption that the fidis converge, that then even a functional limit theorem applies. Secondly, there is no counterpart of Corollary \ref{cont}, which says that $\sigma$ is continuous on $S_u$. For example, let us consider $(C^*,\mathcal{T}_{uc}^*)$ and let $\sigma^*(f)$ be the smallest minimizing point of $f$ (existence assumed). Then one can construct a sequence $(f_n)$ such that $f_n$ converges to $f$ uniformly on every compact $K \subseteq \mathbb{R}$, i.e.,
$f_n \rightarrow f$ in $(C^*,\mathcal{T}_{uc}^*)$, but $\sigma^*(f_n) \rightarrow -\infty$. In particular, $\sigma^*$ is far from being continuous on $S_u$.

\vspace{1cm}
\textbf{Declarations}\\

\textbf{Compliance with Ethical Standards}: I have read and I understand the provided information.\\

\textbf{Competing Interests}: The author has no competing interests to declare that are relevant to the content of
this article.


\begin{thebibliography}{22}

\bibitem{Bill} P. Billingsley, \textit{Convergence of Probability Measures}, New York: John Wiley \& Sons, 1968.

\bibitem{Bill2} P. Billingsley, \textit{Convergence of Probability Measures}, Second Edition, New York: John Wiley \& Sons, 1999.

\bibitem{Davis} R. A. Davis, K. Knight, and J. Liu, \textit{M-estimation for autoregressions with innite variance},
Stochastic Process. Appl. \textbf{40} (1992), 145--180.

\bibitem{Dshal} J. H. Dshalalow, \textit{Real Analysis}, Boca Raton, London, New York, Washington, D.C.: Chapman \& Hall/CRC, 2001.

\bibitem{Embrechts} P. Embrechts and M. Hofert, \textit{A note on generalized inverses}, Mathematical
Methods of Operations Research \textbf{77} (2013), 423--432.

\bibitem{Feng} C. Feng, H. Wang, X. M. Tu and J. Kowalski, \textit{A note on generalized inverses
of distribution function and quantile transformation}, Applied Mathematics \textbf{3} (2012), 2098--2100.

\bibitem{Ferger} D. Ferger, \textit{A Continuous Mapping Theorem for the argmin-set functional with applications to convex stochastic processes},
Kybernetika \textbf{57} (2021), 426--445.

\bibitem{Ferger2} D. Ferger, \textit{Arginf-sets of multivariate cadlag processes and their distributional convergence in hyperspace topologies},
Theory of Stochastic Processes \textbf{20(36)}, No.2 (2015), 13--41.

\bibitem{Ferger3} D. Ferger, \textit{Distributional hyperspace-convergence of Argmin-sets in convex M-estimation},
Theor. Probability and Math. Statist. (2023), in press.

\bibitem{Fortelle} A. de La Fortelle, \textit{A study on generalized inverses and increasing functions Part I: generalized
inverses}, hal-01255512 (2016), 1--14.

\bibitem{Geyer} J. Geyer, \textit{On the asymptotics of convex stochastic optimization}, Unpublished
manuscript (1996).

\bibitem{Hjort} N. L. Hjort and D. Pollard, \textit{Asymptotic for minimizers of convex processes}, Preprint,
Dept. of Statistics, Yale University (1993). arxiv:1107.3806v1

\bibitem{Ibragimov} I. A. Ibragimov and R. Z. Has'minskii, \textit{Statistical Estimation: Asymptotic Theory}, New York: Springer-Verlag, 1981.

\bibitem{Knight} K. Knight, \textit{What are the limiting distributions of quantile estimators?},
In \textit{Statistical Data Analysis Based on the $L_1$-Norm and Related Methods} (Y. Dodge, ed.) 47--65.
Series Statistics for Industry and Technology. Basel: Birkh\"{a}user, 2002.

\bibitem{Stute} P. G\"{a}nssler and W. Stute, \textit{Wahrscheinlichkeitstheorie}, Berlin, Heidelberg, Germany: Springer-Verlag, 1977.

\bibitem{Heuser} H. Heuser, \textit{Lehrbuch der Analysis. Teil 1}, 14-th Edition, Wiesbaden, Germany: B.G. Teubner Verlag, 2004.

\bibitem{Kall1} O. Kallenberg, \textit{Foundations of Modern Probability, Volume 1},
Third Edition, Springer Nature Switzerland AG, 2021.

\bibitem{Kall2} O. Kallenberg, \textit{Foundations of Modern Probability, Volume 2},
Third Edition, Springer Nature Switzerland AG, 2021.

\bibitem{Liese} F. Liese and K.-J. Mieschke, \textit{Statistical Decision Theory}, New York: Springer Science+Business Media, LLC, 2008.

\bibitem{Niculescu}{2006} C. P. Niculescu and L.-E. Persson, \textit{Convex Functions and Their Applications: A Contemporary Approach},
New York: Springer Science+Business Media, Inc., 2006.

\bibitem{Rockafellar} R. T. Rockafellar, \textit{Convex Analysis}, Princeton, New Jersey: Princeton University Press, 1970.

\bibitem{Singh} T. B. Singh, \textit{Introduction to Topology}, Singapore: Springer Nature, 2019.

\bibitem{Top} F. Tops{\o}e, \textit{Topology and Measure}, Lecture Notes in Mathematics Vol. 133, Berlin-Heidelberg-New York: Springer-Verlag, 1970.

\bibitem{Vaart} A. W. van der Vaart and J. A. Wellner, \textit{Weak Convergence and Empirical Processes. With Applications to Statistics}, New York: Springer-Verlag, 1996.

\end{thebibliography}

\end{document}